\documentclass[11pt,reqno,twoside]{amsart}

\usepackage{eurosym}
\usepackage{amsfonts}
\usepackage{amsmath,amssymb}
\usepackage[dvips,bottom=1.4in,right=1in,top=1in, left=1in]{geometry}
\usepackage{epic}
\usepackage{pstricks}
\usepackage{graphicx}
\usepackage{amsmath,amssymb,lineno,amsthm,epic,pstricks,amsthm}

\usepackage{color}

\usepackage{amsfonts,amsmath,amssymb}
\usepackage{mathrsfs,mathtools}
\usepackage{enumerate}
\usepackage{hyperref}
\usepackage{esint}
\usepackage{graphicx}
\usepackage{bm}
\usepackage{commath}
\usepackage{esint}
\DeclareMathAlphabet{\mathpzc}{OT1}{pzc}{m}{it}

\usepackage{color}

\usepackage{eurosym}
\usepackage{amsfonts}
\usepackage{amsmath,amssymb}

\newtheorem{theorem}{Theorem}[section]
\newtheorem{proposition}[theorem]{Proposition}
\newtheorem{lemma}[theorem]{Lemma}

\newtheorem{definition}[theorem]{Definition}

\newtheorem{remark}[theorem]{Remark}

\def\C{{\mathbb C}}

\def\Omc{\RR^N\setminus\Omega}
\def\RR{\mathbb{R}}

\def\NN{\mathbb{N}}

\def\CC{\mathbb{C}}

\def\Om{\Omega}

\def\bOm{\overline{\Om}}

\def\pOm{\partial\Omega}

\numberwithin{equation}{section}

\begin{document}

\title[Approximate controllability]{Approximate controllabilty from the exterior of space-time fractional diffusive equations}

\author{Mahamadi Warma}

\address{University of Puerto Rico, Faculty of Natural Sciences,
Department of Mathematics (Rio Piedras Campus), PO Box 70377 San Juan PR
00936-8377 (USA)}

 \email{mahamadi.warma1@upr.edu, mjwarma@gmail.com}



\begin{abstract}
Let $\Om\subset\RR^N$ be a bounded domain with a Lipschitz continuous boundary. We study the controllability of the space-time fractional diffusive equation
\begin{equation*}
\begin{cases}
\mathbb D_t^\alpha u+(-\Delta)^su=0\;\;&\mbox{ in }\;(0,T)\times\Omega,\\
u=g\chi_{(0,T)\times\mathcal O} &\mbox{ in }\;(0,T)\times(\RR^N\setminus\Omega),\\
u(0,\cdot)=u_0&\mbox{ in }\;\Omega,
\end{cases}
\end{equation*}
where $u=u(t,x)$ is the state to be controlled and $g=g(t,x)$ is the control function which is localized in a non-empty open subset $\mathcal O$ of $\Omc$. 
Here, $0<\alpha\le 1$, $0<s<1$ and $T>0$ are real numbers.
After giving an  explicit representation of solutions, we show that the system is always approximately controllable for every $T>0$, $u_0\in L^2(\Omega)$ and $g\in \mathcal D((0,T)\times\mathcal O)$ where $\mathcal O\subseteq\Omc$ is an arbitrary non-empty open set. The results obtained are sharp in the sense that such a system is never null controllable if $0<\alpha<1$. The proof of our result is based on a new unique continuation principle for the eigenvalues problem associated with the fractional Laplace operator subject to the zero Dirichlet exterior condition that we have established.
\end{abstract}

\maketitle

\section{Introduction}

Let $\Omega\subset\RR^N$ be a bounded domain with a Lipschitz continuous boundary $\pOm$. The main concern in the present paper is to study the controllability of the following exterior-initial value problem
\begin{equation}\label{main-eq}
\begin{cases}
\mathbb D_t^\alpha u+(-\Delta)^su=0\;\;&\mbox{ in }\;(0,T)\times\Omega,\\
u=g\chi_{(0,T)\times\mathcal O} &\mbox{ in }\;(0,T)\times(\RR^N\setminus\Omega),\\
u(0,\cdot)=u_0&\mbox{ in }\;\Omega.
\end{cases}
\end{equation}
In \eqref{main-eq}, $u=u(t,x)$ is the state to be controlled and $g=g(t,x)$ is the control function which is localized in a non-empty open subset $\mathcal O$ of $\RR^N\setminus\Omega$.

Here $0<\alpha\le 1$, $0<s<1$, $T>0$ are real numbers, $\mathbb D_t^\alpha$ denotes the Caputo time fractional derivative of order $\alpha$ (see Definition \ref{Caputo-deriv} below) and $(-\Delta)^s$ is the fractional Laplace operator (see \eqref{fl_def} below).  We refer to Section \ref{sec-pre} for a rigorous definition of the Caputo time fractional derivative and the fractional Laplace operator.

The purpose of the paper is to discuss the controllability properties of the system. 

We shall say that \eqref{main-eq} is approximately controllable in time $T>0$, if for any  $\varepsilon>0$ and $u_0, u_1\in L^2(\Omega)$, there is a control function $g\in L^2((0,T)\times\mathcal O)$ such that the corresponding unique strong solution $u$ of \eqref{main-eq} satisfies
\begin{align*}
\|u(T,\cdot)-u_1\|_{L^2(\Omega)}\le \varepsilon.
\end{align*}

If for every $u_0\in L^2(\Omega)$, there is a control function $g\in L^2((0,T)\times\mathcal O)$ such that the strong solution $u$ of \eqref{main-eq} satisfies
\begin{align*}
u(T,\cdot)=0\;\;\;\mbox{ in }\;\Omega,
\end{align*}
then \eqref{main-eq} is said to be null controllable in time $T>0$.

Giving that the following Dirichlet problem
\begin{equation*}
\begin{cases}
(-\Delta)^su=0\;\;&\mbox{ in }\;\Omega,\\
u=g&\mbox{ on }\pOm,
\end{cases}
\end{equation*}
is not well posed if $0<s<1$, then the study of the controllability of  \eqref{main-eq} cannot be done from the boundary $\pOm$, that is, one cannot replace the second equation in \eqref{main-eq} by $u=g$ on $(0,T)\times \pOm$. In other words our control $g$ cannot be localized on a subset of the boundary $\pOm$.
Since the operator $(-\Delta)^s$, for $0<s<1$, is nonlocal, then knowing $u$ at the boundary $\pOm$ is not enough to know $u$ on all $\RR^N$. The well-posed Dirichlet problem for the fractional Laplacian is given by
\begin{equation*}
\begin{cases}
(-\Delta)^su=0\;\;&\mbox{ in }\;\Omega,\\
u=g&\mbox{ in }\Omc.
\end{cases}
\end{equation*}
For this reason, we think that \eqref{main-eq} gives the right formulation that can replace the classical notion of controllability of PDEs from the boundary. Our objective is to study if \eqref{main-eq} can be approximately controllable from the exterior of $\Omega$. To the best of our knowledge, it is the first work that addresses the controllability of nonlocal equations from the exterior of the domain where the PDE is solved.

First of all, using the results by L\"u and Zuazua \cite{EZ}, one can show that \eqref{main-eq} is never null controllable if $0<\alpha<1$. Therefore it makes sense to study if the system can be approximately controllable from the exterior. The case $\alpha=1$ corresponds to a nonhomogeneous fractional heat type equation. After the present work has been submitted, the interior and the exterior null controllability properties of the one-dimensional fractional heat equation on an internal have been very recently investigated in \cite{BiHe} and \cite{WZ2}, respectively, by using an asymptotic gap condition on the eigenvalues of the fractional Laplace operator with the zero exterior condition on an open interval of $\RR$. 
The cases of the fractional Schr\"odinger and wave equations have been investigated by Biccari \cite{Umb} by using a Pohozaev identity for the fractional Laplacian established in \cite{RS-Po} (see also \cite{RS-Po2}). The interior or/and the exterior null controllabilty of the heat equation associated with the fractional Laplace operator in dimension $N\ge 2$ is still open. The difficulties to study such problems follow from the fact that (but are not limited to) there is still no appropriate Carleman type estimates for the fractional Laplace operator. For this reason and for the seek of completeness we also include the case $\alpha=1$ and investigate if the associated system is approximately controllable.

After giving some results of existence, uniqueness, regularity and explicit representation of solutions to \eqref{main-eq}, we shall show that  it is always approximately controllable, for any $u_0\in L^2(\Omega)$, $T>0$ and any $g\in \mathcal D((0,T)\times\mathcal O)$, where $\mathcal O\subset\Omc$ is an arbitrary non-empty open set. 
To obtain our result, we first establish a unique continuation principle for the eigenvalues problem of the fractional Laplace operator with the zero exterior condition. Using this, in a second step we show that the adjoint system associated with \eqref{main-eq} also satisfies the unique continuation property for evolution equations. We shall obtain the approximate controllability of the system as a consequence of the unique continuation property of the dual system.

We notice that nonlocal partial differential equations are typical models of anomalous diffusion. Space-time fractional diffusion equations have been
used to model anomalous transport in many diverse
disciplines, including finance, semiconductor research, biology,
and hydrogeology \cite{Hilfer}. In the context of flow in
porous media, fractional space derivatives model large
motions through highly conductive layers or fractures,
while fractional time derivatives describe particles that
remain motionless for extended periods of time. 
The scalar
space-fractional diffusion equation governs L\'evy motion,
and the tail parameter $s$  of the L\'evy motion equals the
order of the fractional derivative. Solutions to the vector
space-fractional diffusion equation are operator L\'evy
motions that may scale at different rates in different
directions. Fractional time derivatives are important in reactive
transport, since solutes may interact with the immobile
porous medium in highly nonlinear ways. There is
evidence that solutes may sorb for random amounts of
time that have a power law distribution, or move into
irregularly-sized blocks of relatively immobile water, producing
similar behavior. If the first moment of these
time delays diverges, then a fractional time derivative
applies. 
Anomalous diffusion deviates
from the standard description of Brownian motion, the main character of which is that
its mean squared displacement is a non-linear growth with respect to time.
Some relevant contributions in this field have been achieved. We refer for instance to \cite{ELP}
for a general presentation of nonlocal models and the generating dynamics under the terminology
of peridynamics,  to \cite{Noch} for the  first results of elliptic problems by means of finite element
approximations, to \cite{BeBo} for application to the atmospheric motion induced by the rotation of the
Earth and a description of how these models arise in the context of dune formation and dynamics, and to \cite{MePh} for application in industrial processes like cooling/heating in steel and glass
manufacturing.

Fractional PDEs have also recently received a considerable amount of 
          attention due to their ability to capture long-range correlations
          for instance of material properties and memory effects. One models the former by using
          the fractional Laplacian and the latter by using fractional time derivatives.
          The added advantage of fractional operators is the fact that they 
          impose less smoothness and are therefore more suitable to capture
          heterogeneous and multiscale effects for instance in geophysics 
          and image science 
          \cite{AnBa2017,AnRa2018-Manuscript}. 
          Fractional operators have shown remarkable potential in image denoising 
          to a point that they are competitive (even better) 
           \cite{AnRa2018-Manuscript} than the total variation 
          based approaches \cite{rudin1992nonlinear}.

Control theory is an interdisciplinary branch of engineering and mathematics that deals with influence behavior of dynamical systems. Controllability is one of the fundamental concepts in mathematical control theory. Its main goal is to drive the
solutions of a finite or infinite dimensional dynamical system to rest by means of the action
of an applied control. In the classical concept of controllability associated with local operators, the control function can only be localized inside the domain where the PDE is solved (interior control) or at the boundary of the domain (boundary control). In another words, since null or exact controllability is usually equivalent to an observability property for the associated adjoint system, this notion allows one to detect all the components of such a system based on a partial measurement in the interior or on the boundary.
We have shown above that for nonlocal PDEs associated with the fractional Laplace operator, it turns out that one cannot control or observe such a dynamical system from the boundary. This is due to the fact that since the underlying operator is nonlocal, there is always an interaction between the interior and the exterior of the domain where the PDE is solved. As we have described above, this is the first work where we have shown that for such dynamical systems, a boundary control does not make sense and therefore, should be replaced by an exterior control.
This novel concept of exterior controllability introduced in this paper is applicable  
          whenever a control is placed outside the observation domain (where the PDE is satisfied). 
          This work has been already extended in \cite{HAntil_RKhatri_MWarma_2018a} to optimal control of
          fractional PDEs and external source identification problems 
          and several illustrative computational
          examples. Furthermore, we have also very recently investigated in \cite{LWZ,LR-Wa,WZ1,WZ2} some exterior controllability problems of other dynamical systems such as, space-time super-diffusive equations, strongly damping fractional wave equations and the fractional Moore-Gibson-Thompson equations.
          We emphasize that classical models are not sufficiently rich and they only
          allow source/control placement either inside the observation domain (where the PDE is fulfilled)
          or on the boundary of the domain. Fractional models are sufficiently rich and allows
          new notions of control and source placement. For further details we also refer to \cite{GW-F,GW-CPDE, Go-Ma97, Ma97,Go-Ma00,Meer, Po99,Val} and their references.

In the present paper, we have considered the  Caputo time fractional derivative which has been also used recently to model fractional diffusion in plasma turbulence. Another advantage of using the Caputo derivative in modeling physical problems is that the Caputo derivative of  constant functions is zero. This shows that time-independent solutions are also solutions of the time-dependent problem and this is not the case for the Riemann-Liouville fractional time derivative.  

The rest of the paper is organized as follows. In Section \ref{sec-main} we state the main results of the article. The first one (Theorem \ref{theo-cont-pro}) states that  the adjoint system associated with \eqref{main-eq} satisfies the unique continuation property for evolution equations and the second main result (Theorem \ref{theo-appro-con}) reads that \eqref{main-eq} is always approximately controllable.
 In Section \ref{sec-pre} we introduce the function spaces needed to study our problem and we prove some intermediate results concerning the Dirichlet problem for the fractional Laplacian that will be used throughout the paper. In particular, we establish the unique continuation principle (Theorem \ref{theo-27}) for the eigenvalues problem of $(-\Delta)^s$ with the exterior condition $u=0$ in $\Omc$.
This is followed by the proof of the existence, uniqueness, regularity, the explicit representation of solutions of \eqref{main-eq} and its associated adjoint system (Theorem \ref{theo-weak} and Proposition \ref{Dual-theo-weak}) in Section \ref{sec-proof-in}. Finally in Section \ref{sec-proof-res}, we give the proof of the main results stated in Section \ref{sec-main}.

\section{The main results}\label{sec-main}

In this section we state the main results of the paper. Throughout the rest of the article, without any mention $0<s<1$ and $0<\alpha\le 1$ are fixed real numbers.
First, we recall our notion of approximately controllable. 

\begin{definition}
We shall say that the system \eqref{main-eq} is approximately controllable at time $T>0$, if for any $u_0, u_1\in L^2(\Omega)$ and $\varepsilon>0$, there is a control function $g$ such that the corresponding unique strong solution $u$ of \eqref{main-eq} satisfies
\begin{align}\label{def-app}
\|u(T,\cdot)-u_1\|_{L^2(\Omega)}\le \varepsilon.
\end{align}
\end{definition}

Next, let $u_0\in L^2(\Omega)$ and consider the following two systems:
\begin{equation}\label{Int-eq}
\begin{cases}
\mathbb D_t^\alpha u+(-\Delta)^su=0\;\;&\mbox{ in }\;(0,T)\times\Omega,\\
u=g &\mbox{ in }\;(0,T)\times(\RR^N\setminus\Omega),\\
u(0,\cdot)=0&\mbox{ in }\;\Omega,
\end{cases}
\end{equation}
and
\begin{equation}\label{Int-eq-2}
\begin{cases}
\mathbb D_t^\alpha w+(-\Delta)^sw=0\;\;&\mbox{ in }\;(0,T)\times\Omega,\\
w=0 &\mbox{ in }\;(0,T)\times(\RR^N\setminus\Omega),\\
w(0,\cdot)=u_0&\mbox{ in }\;\Omega.
\end{cases}
\end{equation}

Here, we shall use the fractional order Sobolev spaces  $W^{s,2}(\RR^N)$, $W_0^{s,2}(\bOm)$, their dual $W^{-s,2}(\RR^N)$, $W_0^{-s,2}(\bOm)$ that will be introduced in Section \ref{sec-32}, and we shall denote by $\langle\cdot,\cdot\rangle$, their duality map if there is no confusion.

The following is our notion of solution.

\begin{definition}
A function $u$ is said to be a strong solution of \eqref{Int-eq}, if for a.e. $t\in (0,T)$ and for every $T>0$, the following properties holds.
\begin{itemize}
\item Regularity:
\begin{equation*}
u\in C([0,T];L^2(\Omega)),\;\;\;\;\mathbb D_t^\alpha u\in C((0,T];W^{-s,2}(\RR^N)),
\end{equation*}
and \eqref{Int-eq} holds in $W^{-s,2}(\RR^N)$ for a.e. $t\in (0,T)$. That is, for every $v\in W^{s,2}(\RR^N)$ and a.e. $t\in (0,T)$, we have
\begin{align*}
\langle\mathbb D_t^\alpha u,v\rangle +\langle (-\Delta)^su,v\rangle=0.
\end{align*}
\item Initial and exterior conditions:
\begin{equation*}
u(0,\cdot)=0\;\;\mbox{ in }\;\Omega\;\mbox{ and }\; u=g\;\mbox{ in }\;(0,T)\times(\Omc).
\end{equation*}
\end{itemize}
\end{definition}

\begin{remark}\label{rem-33}
{\em
We mention the following facts.
\begin{enumerate}
\item Since  \eqref{Int-eq-2} can be rewritten as the following Cauchy problem:
\begin{equation*}
\mathbb D_t^\alpha w+(-\Delta)_D^sw=0\;\mbox{ in }\;(0,T)\times\Omega,\;\;\; w(0,\cdot)=u_0\;\mbox{ in }\;\Omega,
\end{equation*}
it follows that  (see e.g. \cite{GW-F,KLW} and their references)  for every $u_0\in L^2(\Omega)$, there is a $w\in C([0,T];L^2(\Omega))$ which is the unique strong solution of \eqref{Int-eq-2}. Here $(-\Delta)_D^s$ denotes the realization in $L^2(\Omega)$ of $(-\Delta)^s$ with the zero exterior condition $u=0$ in $\Omc$ (see Section \ref{sec-33}).

\item Assume that \eqref{Int-eq} is approximately controllable in the sense of \eqref{def-app} and let $u_1$ be the approximated function. Then for every $\varepsilon>0$, there is a control function $g$ such that the solution $u$ of \eqref{Int-eq} satisfies
\begin{align}\label{eq-app}
\|u(T,\cdot)-(u_1-w(T,\cdot))\|_{L^2(\Omega)}\le \varepsilon.
\end{align}
By definition, we have that $u+w$ solves \eqref{main-eq} and it follows from \eqref{eq-app} that
\begin{align*}
\|(u+w)(T,\cdot)-u_1\|_{L^2(\Omega)}\le \varepsilon.
\end{align*}
Thus \eqref{main-eq} is approximately controllable. In view of this property, in our study we shall consider \eqref{Int-eq} instead of \eqref{main-eq}.
\end{enumerate}
}
\end{remark}

Next, using the integration by parts formula (see \eqref{IP01}) we have that the adjoint system associated with \eqref{Int-eq} is given by
\begin{equation}\label{dual}
\begin{cases}
D_{t,T}^\alpha v+(-\Delta)^sv=0\;\;&\mbox{ in }\;(0,T)\times\Omega,\\
v=0 &\mbox{ in }\;(0,T)\times(\RR^N\setminus\Omega),\\
I_{t,T}^{1-\alpha}v(T,\cdot)=u_0&\mbox{ in }\;\Omega,
\end{cases}
\end{equation}
where $D_{t,T}^\alpha$ (resp. $I_{t,T}^{1-\alpha}$) denotes the right Riemann-Liouville time fractional derivative of order $\alpha$ (resp. the right Riemann-Liouville time fractional integral of order $1-\alpha$) that will be introduced in Section \ref{sec-31}.

\begin{definition}
A function $v$ is said to be a strong solution of \eqref{dual}, if for a.e. $t\in (0,T)$ and for every $T>0$, the following properties holds.
\begin{itemize}
\item Regularity:
\begin{equation*}
 I_{t,T}^{1-\alpha}v\in C([0,T];L^2(\Om)),\;\;\;D_{t,T}^\alpha v
\in C((0,T);W^{-s,2}(\bOm)),
\end{equation*}
and \eqref{dual} holds in $W^{-s,2}(\bOm)$ for a.e. $t\in (0,T)$. That is, for every $u\in W_0^{s,2}(\bOm)$ and a.e. $t\in (0,T)$, we have
\begin{align*}
\langle D_{t,T}^\alpha v,u\rangle +\langle (-\Delta)^sv,u\rangle=0.
\end{align*}
\item Final condition:
\begin{equation*}
I_{t,T}^{1-\alpha}v(T,\cdot)=u_0\;\;\mbox{ in }\;\Omega.
\end{equation*}
\end{itemize}
\end{definition}

Our first main result shows that \eqref{dual} satisfies the unique continuation property.

\begin{theorem}\label{theo-cont-pro}
Let $u_0\in L^2(\Omega)$ and $v$  the unique strong solution of \eqref{dual}. Let $\mathcal N_s$ denote the nonlocal normal derivative (see \eqref{NLND}) and let
$\mathcal O\subseteq\Omc$ be an arbitrary non-empty open set.
If $\mathcal N_sv=0$ in $(0,T)\times\mathcal O$, then $v=0$ in $(0,T)\times\Omega$.
\end{theorem}

The second main result is a direct consequence of Theorem \ref{theo-cont-pro}.

\begin{theorem}\label{theo-appro-con}
The system \eqref{main-eq} is approximately controllable at any time $T>0$, for every $u_0\in L^2(\Omega)$ and $g\in \mathcal D((0,T)\times\mathcal O)$, where $\mathcal O\subseteq\Omc$ is an arbitrary non-empty open.  That is,
\begin{align*}
\overline{\Big\{u(T,\cdot):\; g\in \mathcal D((0,T)\times\mathcal O)\Big\}}^{L^2(\Omega)}=L^2(\Omega),
\end{align*}
where $u$ is the unique strong solution of \eqref{main-eq} with initial data $u_0$.
\end{theorem}

\section{Preliminary results}\label{sec-pre}

In this section we give some notations, introduce the function spaces needed to investigate our problem and we also prove some important intermediate results that will be used throughout the paper.

\subsection{Fractional in time derivatives and the Mittag-Leffler functions}  \label{sec-31}
We first recall the notion of fractional-in-time derivative in the sense of
Caputo and Riemann-Liouville. Let $\alpha \in \left( 0,1\right) $ and define%
\begin{equation*}
g_{\alpha }\left( t\right) := 
\begin{cases}
\displaystyle\frac{t^{\alpha -1}}{\Gamma (\alpha )} & \text{if }t>0, \\ 
0 & \text{if }t\leq 0,
\end{cases}
\label{g_alpha}
\end{equation*}
where $\Gamma$ is the usual Gamma function. It is convenient to denote $g_0:=\delta_0$, the Dirac measure concentrated at the point $0$. Throughout this subsection, $T>0$ is a real number, $X$ is a Banach space and we consider functions defined from $(0,T)$ into $X$.

\begin{definition}
\label{Caputo-deriv}Let $f\in C\left(
[0,T];X\right) $ such that $(g_{1-\alpha }\ast f)\in W^{1,1}\left( (0,T);X\right)
. $ The {\bf Riemann-Liouville fractional derivative of order $\alpha$} is given for almost all $t\in (0,T)$ by
\begin{equation*}
D_{t}^{\alpha }f\left( t\right) :=\frac{d}{dt}\Big(g_{1-\alpha }\ast f\Big)%
(t)=\frac{d}{dt}\int_{0}^{t}g_{1-\alpha }\left( t-\tau \right) f\left(
\tau\right) d\tau.
\end{equation*}
We define the {\bf fractional derivative of order $
\alpha ,$ of Caputo-type}, as follows:%
\begin{equation}
\mathbb D_{t}^{\alpha }f\left( t\right) :=D_{t}^{\alpha }\left( f\left(
t\right) -f\left( 0\right) \right) ,\text{ for a.e. }t\in (0,T].  \label{fract}
\end{equation}
\end{definition}

We notice that \eqref{fract} in Definition \ref{Caputo-deriv} gives a weaker
notion of (Caputo) fractional derivative compared to the original definition
introduced by Caputo in the late 1960s (see \cite{Cap}) which reads
\begin{align}\label{Cap-de}
\partial_t^\alpha f(t)=(g_{1-\alpha}\ast \partial_tf)(t).
\end{align}
In particular, \eqref{fract} does not require $f$ to be differentiable. In addition we have that 
$\mathbb D_{t}^{\alpha }\left( c\right) =0,$ for any constant $c$. For these
reasons, \eqref{fract} offers a better alternative than the classical notion
of Caputo derivative given in \eqref{Cap-de}. We refer to  \cite[Proposition 2.34]{PN}
which shows the two notions coincide when $f$ is smooth enough, namely, 
\begin{equation*} 
\mathbb D_{t}^{\alpha }f=\partial_t^\alpha f=g_{1-\alpha }\ast \partial _{t}f,\;\mbox{ for }%
\;f\in C^{1}\left( [0,T];X\right) . 
\end{equation*}%
In the classical case when $\alpha
=1,$ we let $\mathbb D_{t}^{1}:=d/dt\left( =\partial _{t}\right) $.

The {\bf right Riemann-Liouville fractional integral of order $\alpha$} is defined as
\begin{align*}
I_{t,T}^\alpha f(t)=\frac{1}{\Gamma(\alpha)}\int_t^T (\tau-t)^{\alpha-1}f(\tau)\;d\tau.
\end{align*}
The {\bf right Riemann-Liouville fractional derivative of order $\alpha$} is defined by

\begin{align}\label{D-R}
D_{t,T}^\alpha u(t)=-\frac{d}{dt}\left(I_{t,T}^{1-\alpha} f\right)(t)
=\frac{-1}{\Gamma(1-\alpha)}\frac{d}{dt}\int_t^T(\tau-t)^{-\alpha}f(\tau)\;d\tau.
\end{align}
From \eqref{D-R}, we have that if $f$ is differentiable, then $D_{t,T}^1f=-\partial_tf$.

The {\bf left Riemann-Liouville fractional integral of order $\alpha$} is defined as
\begin{align}\label{l-int}
I_{0,t}^\alpha f(t)=\frac{1}{\Gamma(\alpha)}\int_0^t (t-\tau)^{\alpha-1}f(\tau)\;d\tau.
\end{align}
It follows from \eqref{Cap-de} and \eqref{l-int} that for every $f\in W^{1,1}((0,T);X)$, we have that
\begin{align}\label{eq33}
\partial_t^\alpha f(t)=I_{0,t}^{1-\alpha}\partial_tf(t).
\end{align}

The following {\bf integration by parts formula} is taken from \cite{Agr}. Let $0<\alpha\le1$. Then
\begin{align}\label{IP01}
\int_0^Tg(t)\mathbb D_{t}^\alpha f(t)\;dt=&\int_0^Tf(t)D_{t,T}^\alpha g(t)\;dt
+\left[f(t)I_{t,T}^{1-\alpha}g(t)\right]_{t=0}^{t=T},
\end{align}
provided that the left and right-hand sides expressions make sense.

The following result is contained in \cite[Corollary, pp. 67]{Ma}.

\begin{lemma}
Let $0<\alpha\le 1$ and $1<p,q<\infty$ be such that $\frac 1p+\frac 1q\le 1+\alpha$, with $p\ne 1$, $q\ne 1$ in the case $\frac 1p+\frac 1q= 1+\alpha$. If $f\in L^p((0,T);X)$ and $g\in L^q((0,T);X)$, then
\begin{align}\label{eq35}
\Big((I_{0,t}^\alpha f)\ast g\Big)(t)=\Big(f\ast(I_{0,t}^\alpha g)\Big)(t),\;\;\;t\in (0,T).
\end{align}
\end{lemma}

The Mittag-Leffler function with two parameters is defined as follows:
\begin{align}\label{mm}
E_{\alpha, \beta}(z) := \sum_{n=0}^{\infty}\frac{z^n}{\Gamma(\alpha n + \beta)},\; \;\alpha>0,\;\beta \in\C, \quad z \in \C.
\end{align}
It is clear that $E_{1,1}(z)=e^z$ and that $E_{\alpha,\beta}(z)$ is an entire function.
The following estimate of $E_{\alpha,\beta}(z)$ will be useful. Let $\alpha>0$, $\beta\in\RR$ and $\mu$ be such that $\frac{\alpha\pi}{2}<\mu<\min\{\pi,\alpha\pi\}$. Then there is a constant $C=C(\alpha,\beta,\mu)>0$ such that
\begin{equation}\label{Est-MLF}
|E_{\alpha,\beta}(z)|\le \frac{C}{1+|z|},\;\;\;\mu\le |\mbox{arg}(z)|\le \pi.
\end{equation}
In the literature, frequently the notation $E_\alpha=E_{\alpha, 1}$ is used.

The Laplace transform  of the Mittag-Leffler function is given by:
\begin{equation}\label{lap-ml}
\int_0^{\infty} e^{-\lambda t} t^{\alpha  + \beta-1}
E_{\alpha,\beta}(-\omega t^{\alpha})dt = \frac{
\lambda^{\alpha-\beta}}{(\lambda^{\alpha} + \omega)^{2}},
\quad \mbox{Re}(\lambda)> |\omega|^{\frac{1}{\alpha}}.
\end{equation}
If $\lambda>0$, then
\begin{equation}\label{Est-MLF2}
\frac{d}{dt}\Big(E_{\alpha,1}(-\lambda t^\alpha)\Big)=-\lambda t^{\alpha-1}E_{\alpha,\alpha}(-\lambda t^\alpha),\;\;t>0.
\end{equation}
Using integration we easily get that for every $0<\alpha\le 1$ and $\lambda>0$,
\begin{align}\label{eq34}
I_{0,t}^{1-\alpha}\left(t^{\alpha-1}E_{\alpha,\alpha}(-\lambda t^\alpha)\right)=E_{\alpha,1}(-\lambda t^\alpha),\;\;t>0.
\end{align}
The proofs of \eqref{Est-MLF} and \eqref{Est-MLF2} are contained in \cite[Chapter 1]{Po99}. The proofs of \eqref{lap-ml} and \eqref{eq34} are contained in \cite{KLW} and \cite{KW}, respectively.

In order to show the regularity of strong solutions of \eqref{main-eq}, we shall frequently use the following estimates that follow from \eqref{Est-MLF} and a straightforward computation.

\begin{lemma}\label{lem-INE}
Let $0<\alpha\le 1$, $0\le\gamma<1$, $\lambda>0$ and $\beta>0$ be real numbers. Then the following assertions hold.
\begin{enumerate}
\item There is a constant $C>0$, such that for every $t>0$,
\begin{align}\label{IN-L2}
\left|\lambda^{1-\gamma}t^{\alpha-1} E_{\alpha,\beta}(-\lambda t^\alpha)\right|\le Ct^{\alpha\gamma-1}.
\end{align}

\item There is a constant $C>0$, such that for every $t>0$,
\begin{equation}\label{New}
\left|\lambda^{1-\gamma} E_{\alpha,1}(-\lambda t^\alpha)\right|\le Ct^{\alpha(\gamma-1)}.
\end{equation}
\end{enumerate}
\end{lemma}

For more details on fractional order derivatives, integrals and the  Mittag-Leffler functions  we refer to \cite{Agr, Ba01,Go-Ma97,Go-Ma00,Ma97,Mi-Ro,Po99} and the references therein.

\subsection{Fractional order Sobolev spaces and the fractional Laplacian} \label{sec-32}
For $0<s<1$ and $\Omega\subseteq\RR^N$ an arbitrary open set, we let
\begin{align*}
W^{s,2}(\Omega):=\left\{u\in L^2(\Omega):\;\int_{\Omega}\int_{\Omega}\frac{|u(x)-u(y)|^2}{|x-y|^{N+2s}}\;dxdy<\infty\right\},
\end{align*}
and we endow it with the norm defined by
\begin{align*}
\|u\|_{W^{s,2}(\Omega)}=\left(\int_{\Omega}|u(x)|^2\;dx+\int_{\Omega}\int_{\Omega}\frac{|u(x)-u(y)|^2}{|x-y|^{N+2s}}\;dxdy\right)^{\frac 12}.
\end{align*}
We set
\begin{align*}
W_0^{s,2}(\bOm):=\Big\{u\in W^{s,2}(\RR^N):\;u=0\;\mbox{ in }\;\RR^N\setminus\Omega\Big\}.
\end{align*}
Letting $\mathcal D(\Omega)$ denote the space of all continuously infinitely differentiable functions with compact support in $\Omega$, then $\mathcal D(\Omega)$ is not always dense in $W_0^{s,2}(\bOm)$. But if $\Omega$ has a Lipschitz continuous boundary, then $\mathcal D(\Omega)$ is dense in $W_0^{s,2}(\bOm)$ (see e.g. \cite{Val-De}).

We shall denote by $W^{-s,2}(\RR^N)$ and $W^{-s,2}(\bOm)$ the dual of $W^{s,2}(\RR^N)$ and $W_0^{s,2}(\bOm)$, respectively, and by $\langle\cdot,\cdot\rangle$, their duality map if there is no confusion.

We let
\begin{align}\label{local}
W_{\rm loc}^{s,2}(\Omega):=\Big\{u\in L^2(\Omega):\; u\varphi\in W^{s,2}(\Omega),\;\forall\;\varphi\in\mathcal D(\Omega)\Big\}.
\end{align}
For more information on fractional order Sobolev spaces, we refer to \cite{NPV,Gris,JW,War} and their references.

Next, we give a rigorous definition of the fractional Laplace operator. Let 
\begin{align*}
\mathcal L_s^{1}(\RR^N):=\left\{u:\RR^N\to\RR\;\mbox{ measurable},\; \int_{\RR^N}\frac{|u(x)|}{(1+|x|)^{N+2s}}\;dx<\infty\right\}.
\end{align*}
For $u\in \mathcal L_s^{1}(\RR^N)$ and $\varepsilon>0$ we set
\begin{align*}
(-\Delta)_\varepsilon^s u(x):= C_{N,s}\int_{\{y\in\RR^N:\;|x-y|>\varepsilon\}}\frac{u(x)-u(y)}{|x-y|^{N+2s}}\;dy,\;\;x\in\RR^N,
\end{align*}
where $C_{N,s}$ is a normalization constant, given by
\begin{align*}
C_{N,s}:=\frac{s2^{2s}\Gamma\left(\frac{2s+N}{2}\right)}{\pi^{\frac
N2}\Gamma(1-s)}.
\end{align*}
The {\bf fractional Laplacian}  $(-\Delta)^s$ is defined by the following singular integral:
\begin{align}\label{fl_def}
(-\Delta)^su(x)=C_{N,s}\,\mbox{P.V.}\int_{\RR^N}\frac{u(x)-u(y)}{|x-y|^{N+2s}}\;dy=\lim_{\varepsilon\downarrow 0}(-\Delta)_\varepsilon^s u(x),\;\;x\in\RR^N,
\end{align}
provided that the limit exists. 
We notice that if $0<s<1/2$ and $u$ is smooth, for example bounded and Lipschitz continuous on $\RR^N$, then the integral in \eqref{fl_def} is in fact not really singular near $x$ (see e.g. \cite[Remark 3.1]{NPV}). Moreover,  $\mathcal L_s^{1}(\RR^N)$ is the right space for which $ v:=(-\Delta)_\varepsilon^s u$ exists for every $\varepsilon>0$, $v$ being also continuous at the continuity points of  $u$.  

For more details on the fractional Laplace operator we refer to \cite{BBC,Caf1,Caf3,NPV,GW-CPDE,War,War-In} and their references.

\subsection{The Dirichlet problem for the fractional Laplacian}\label{sec-33}
Now assume that $\Omega$ is a bounded domain with a Lipschitz continuous boundary and consider the following Dirichlet problem:
\begin{equation}\label{EDP}
\begin{cases}
(-\Delta)^su=0\;\;&\mbox{ in }\;\Omega,\\
u=g&\mbox{ in }\;\RR^N\setminus\Omega.
\end{cases}
\end{equation}

\begin{definition}\label{def-sol}
Let $g\in W^{s,2}(\Omc)$ and $\widetilde g\in W^{s,2}(\RR^N)$ be such that $\widetilde g|_{\Omc}=g$. A $u\in W^{s,2}(\RR^N)$ is said to be a weak solution of \eqref{EDP} if $u-\widetilde g\in W_0^{s,2}(\bOm)$ and 
\begin{align*}
\int_{\RR^N}\int_{\RR^N}\frac{(u(x)-u(y))(v(x)-v(y))}{|x-y|^{N+2s}}\;dxdy=0,
\end{align*}
for every $v\in W_0^{s,2}(\bOm)$.
\end{definition}

The following result is taken from \cite{Grub} (see also \cite{GSU,VE}).

\begin{proposition}
For every $g\in W^{s,2}(\RR^N\setminus\Omega)$, there is a unique $u\in W^{s,2}(\RR^N)$ satisfying \eqref{EDP} in the sense of Definition \ref{def-sol}. In addition, there is a constant $C>0$ such that
\begin{align}\label{E-Neu}
\|u\|_{W^{s,2}(\RR^N)}\le C\|g\|_{W^{s,2}(\RR^N\setminus\Omega)}.
\end{align}
\end{proposition}

Throughout the following, for $g\in W^{s,2}(\Omc)$, we shall denote by $U_g$ the unique weak solution of  \eqref{EDP}.

Next, we consider the dual system to  \eqref{EDP}, that is, the Dirichlet problem
\begin{equation}\label{EDP-D}
\begin{cases}
(-\Delta)^sv=f\;\;&\mbox{ in }\;\Omega,\\
v=0&\mbox{ in }\;\RR^N\setminus\Omega.
\end{cases}
\end{equation}
It is known (see e.g. \cite{RS-DP}) that for every $f\in L^2(\Omega)$, there is a $v\in W_0^{s,2}(\bOm)$ which is the unique weak solution of \eqref{EDP-D} in the sense that
\begin{align*}
\frac{C_{N,s}}{2}\int_{\RR^N}\int_{\RR^N}\frac{(v(x)-v(y))(w(x)-w(y))}{|x-y|^{N+2s}}\;dxdy=\int_{\Omega}fw\;dx,
\end{align*}
for every $w\in W_0^{s,2}(\bOm)$. 

Throughout the rest of the article, for $f\in L^2(\Om)$, we shall denote by $V^f$ the unique weak solution of \eqref{EDP-D}.

Next, we consider the realization of $(-\Delta)^s$ in $L^2(\Omega)$ with the zero exterior condition $u=0$ in $\Omc$. More precisely, we consider the closed and bilinear form
\begin{align*}
\mathcal F(u,v):=\frac{C_{N,s}}{2}\int_{\RR^N}\int_{\RR^N}\frac{(u(x)-u(y))(v(x)-v(y))}{|x-y|^{N+2s}}\;dxdy,\;\;u,v\in W_0^{s,2}(\bOm).
\end{align*}
Let $(-\Delta)_D^s$ be the selfadjoint operator on $L^2(\Omega)$ associated with $\mathcal F$ in the sense that
\begin{equation*}
\begin{cases}
D((-\Delta)_D^s)=\Big\{u\in W_0^{s,2}(\bOm),\;\exists\;f\in L^2(\Omega),\;\mathcal F(u,v)=(f,v)_{L^2(\Omega)}\;\forall\;v\in W_0^{s,2}(\bOm)\Big\}\\
(-\Delta)_D^su=f.
\end{cases}
\end{equation*}
Then  $(-\Delta)_D^s$ is given by
\begin{align*}
D((-\Delta)_D^s)=\Big\{u\in W_0^{s,2}(\bOm):\; (-\Delta)^su\in L^2(\Omega)\Big\}, \; (-\Delta)_D^su=(-\Delta)^su.
\end{align*}
The operator $(-\Delta)_D^s$ has a compact resolvent and hence, a discrete spectrum formed with eigenvalues satisfying
\begin{align*}
0<\lambda_1\le \lambda_2\le\cdots\le\lambda_n\le\cdots\;\;\mbox{ with }\;\lim_{n\to\infty}\lambda_n=\infty.
\end{align*}
We shall denote by $\varphi_n$ the normalized eigenfunctions associated with  $\lambda_n$. We notice that $\varphi_n$ satisfies the system
\begin{equation}\label{ei-val-pro}
\begin{cases}
(-\Delta)^s\varphi_n=\lambda_n\varphi_n\;\;&\mbox{ in }\;\Omega,\\
\varphi_n=0\;&\mbox{ in }\;\Omc,
\end{cases}
\end{equation}
and $(\varphi_n)$ is total in $L^2(\Omega)$.

We can also introduce the fractional powers of $(-\Delta)_D^s$ as follows. For every $\gamma\ge 0$ we define 
\begin{align*}
\mathbb V_{s,\gamma}:=D([(-\Delta)_D^s]^\gamma)=\left\{u\in L^2(\Omega):\;\sum_{n=1}^\infty|\lambda_n^\gamma(u,\varphi_n)_{L^2(\Omega)}|^2<\infty\right\},
\end{align*}
and for $u\in\mathbb V_{s,\gamma}$ we set
\begin{align*}
[(-\Delta)_D^s]^\gamma u=\sum_{n=1}^\infty\lambda_n^\gamma(u,\varphi_n)_{L^2(\Omega)}.
\end{align*}
Let us notice that $[(-\Delta)_D^s]^\gamma$ does not coincide with $(-\Delta)_D^{s\gamma}$.
For $u\in \mathbb V_{s,\gamma}$, one has 
\begin{align*}
\|u\|_{\mathbb V_{s,\gamma}}^2=\left\|[(-\Delta)_D^s]^\gamma u\right\|_{L^2(\Omega)}^2=\sum_{n=1}^\infty|\lambda_n^\gamma(u,\varphi_n)_{L^2(\Omega)}|^2.
\end{align*}
We mention that contrary to the Laplace operator on smooth open sets where one has maximal elliptic regularity, it is known that for $(-\Delta)^s$, in general $D((-\Delta)_D^s)\not\subset W^{2s,2}(\Omega)$. More precisely, assume that $\Omega$ is smooth then we have the following.  If $0<s<\frac 12$, then by \cite[Formula (7.4)]{Grub},  $D((-\Delta)_D^s)= W_0^{2s,2}(\Omega)$. But if $\frac 12\le s<1$, an example has been given in \cite[Remark 7.2]{RS-DP} where $D((-\Delta)_D^s)\not\subset W^{2s,2}(\Omega)$. We also refer to \cite{BWZ2,BWZ1,RS-DP} for more details on some local regularity results. We also notice that if $0<\gamma<1$, then $\mathbb V_{s,\gamma}=[D((-\Delta)_D^s),L^2(\Omega)]_{1-\gamma}$, the complex interpolation space.
Recall that we have the continuous embedding
\begin{align}\label{27-2}
D((-\Delta)_D^s)  \hookrightarrow W_0^{s,2}(\bOm) \hookrightarrow W^{s,2}(\RR^N).
\end{align}
 Exploiting \eqref{27-2}, we get that if $0<\gamma< \frac 14$, then 
\begin{equation}\label{27}
\mathbb V_{s,1-\gamma}\hookrightarrow W^{s,2}(\RR^N),
\end{equation}
and 
\begin{align}\label{27-1}
W_0^{s,2}(\bOm)\hookrightarrow \mathbb V_{s,\gamma}.
\end{align}

Throughout the following, for a measurable set $E\subseteq\RR^N$, we shall denote by $(\cdot,\cdot)_{L^2(E)}$ the scalar product in $L^2(E)$.

From now on, without any mention, we assume that $\Omega\subset\RR^N$ is a bounded domain with a Lipschitz continuous boundary.  

\subsection{The unique continuation principle}\label{sec-34}

For $u\in W^{s,2}(\RR^N)$ we introduce the {\bf nonlocal normal derivative $\mathcal N_s$} given by 
\begin{align}\label{NLND}
\mathcal N_su(x):=C_{N,s}\int_{\Omega}\frac{u(x)-u(y)}{|x-y|^{N+2s}}\;dy,\;\;\;x\in\RR^N\setminus\bOm.
\end{align}
Since equality has to be understood in a.e., we have that \eqref{NLND} is the same as for a.e. $x\in\Omc$.

The following result is taken from \cite[Lemma 3.2]{GSU}.

\begin{lemma}\label{lem-GSU}
The operator  $\mathcal N_s$ maps $W^{s,2}(\RR^N)$ into $W_{\rm loc}^{s,2}(\Omc)$.
\end{lemma}

It follows from Lemma \ref{lem-GSU} and \eqref{local}, that if $u\in W^{s,2}(\RR^N)$, then the function $\mathcal N_su\in L^2(\Omc)$.
Using this fact, \cite[Lemma 3.3]{DRV}, and a density argument we get the following result which will play an important role in the proof of our main results.

\begin{proposition}
 Let $u\in W^{s,2}(\RR^N)$ be such that $(-\Delta)^su\in L^2(\Omega)$. Then for every $v\in W^{s,2}(\RR^N)$, the identity
\begin{align}\label{Int-Part}
\frac{C_{N,s}}{2}&\int\int_{\RR^{2N}\setminus(\Omc)^2}\frac{(u(x)-u(y))(v(x)-v(y))}{|x-y|^{N+2s}}\;dxdy\notag\\
=&\int_{\Omega}v(-\Delta)^su\;dx+\int_{\Omc}v\mathcal N_su\;dx,
\end{align}
holds, where
\begin{align*}
\RR^{2N}\setminus(\Omc)^2=(\Omega\times\Omega)\cup(\Omega\times(\Omc))\cup((\Omc)\times\Omega).
\end{align*}
\end{proposition}

We have the following result.
\begin{lemma}
Let $g\in W^{s,2}(\Omc)$ and let $U_g\in W^{s,2}(\RR^N)$ be the unique weak solution of \eqref{EDP}. Then  for every $f\in L^2(\Omega)$, we have that
\begin{align}\label{sol-equi}
\int_{\Omega}fU_g\;dx+\int_{\Omc}g\mathcal N_sV^f\;dx=0,
\end{align}
where we recall that $V^f\in W_0^{s,2}(\bOm)$ is the unique weak solution of \eqref{EDP-D}.
\end{lemma}

\begin{proof}
Since $U_g\in W^{s,2}(\RR^N)$,  $(-\Delta)^sU_g=0$ in $\Omega$ and $V^f\in  W^{s,2}(\RR^N)$ with $V^f=0$ in $\Omc$, it follows from \eqref{Int-Part} that
\begin{align}\label{eq26}
\frac{C_{N,s}}{2}&\int\int_{\RR^{2N}\setminus(\Omc)^2}\frac{(U_g(x)-U_g(y))(V^f(x)-V^f(y))}{|x-y|^{N+2s}}\;dxdy\notag\\
&=\int_{\Omega}V^f(-\Delta)^sU_g\;dx+\int_{\Omc}V^f\mathcal N_sU_g\;dx=0.
\end{align}

Similarly, since $V^f\in W^{s,2}(\RR^N)$,  $(-\Delta)^s V^f=f\in L^2(\Omega)$, and $U_g\in W^{s,2}(\RR^N)$, it follows from \eqref{Int-Part} that
\begin{align}\label{eq27}
\frac{C_{N,s}}{2}&\int\int_{\RR^{2N}\setminus(\Omc)^2}\frac{(U_g(x)-U_g(y))(V^f(x)-V^f(y))}{|x-y|^{N+2s}}\;dxdy\notag\\
&=\int_{\Omega}U_g(-\Delta)^sV^f\;dx+\int_{\Omc}U_g\mathcal N_sV^f\;dx\notag\\
&=\int_{\Omega}fU_g\;dx+\int_{\Omc}g\mathcal N_sV^f\;dx.
\end{align}
Subtracting \eqref{eq26} and \eqref{eq27} we get \eqref{sol-equi}. The proof is finished.
\end{proof}

We notice that since $V^f=0$ in $\Omc$, it follows that 
\begin{align*}
&\int\int_{\RR^{2N}\setminus(\Omc)^2}\frac{(U_g(x)-U_g(y))(V^f(x)-V^f(y))}{|x-y|^{N+2s}}\;dxdy\notag\\
=&\int_{\RR^N}\int_{\RR^{N}}\frac{(U_g(x)-U_g(y))(V^f(x)-V^f(y))}{|x-y|^{N+2s}}\;dxdy.
\end{align*}

\begin{remark}
{\em
We mention the following facts.
\begin{enumerate}
\item If in \eqref{sol-equi}, one takes $f=\lambda_n\varphi_n$, hence, $V^f=\varphi_n$, then we get the identity
\begin{align}\label{eqA9}
\lambda_n\int_{\Omega}\varphi_nU_g\;dx+\int_{\Omc}g\mathcal N_s\varphi_n\;dx=0.
\end{align}
\item Since the operator $(-\Delta)_D^s$ is invertible, we have that for every $f\in L^2(\Om)$, the solution $V^f$ of  \eqref{EDP-D} is given by $V^f=((-\Delta)_D^s)^{-1}f$.
\end{enumerate}
}
\end{remark}
For more details we refer to \cite{BBC,DRV,GW-CPDE,Grub,RS2-2,War} and the references therein.

Next, let us denote by $\mathbb P$ the operator defined by
\begin{align*}
\mathbb P:W^{s,2}(\Omc)\to W^{s,2}(\RR^N):\; g\mapsto \mathbb Pg:=U_g,
\end{align*}
where $U_g$ is the unique weak  solution of \eqref{EDP}.

We have the following unique continuation principle which is the main result of this section. It will play a crucial role in the proof of our main results.

\begin{theorem}\label{theo-27}
Let $\mathcal O\subset \Omc$ be an arbitrary non-empty open set. Let $\lambda>0$ and let $\varphi\in D((-\Delta)_D^s)$ satisfy
\begin{equation}\label{e217}
\begin{cases}
(-\Delta)_D^s\varphi=\lambda\varphi\;\;\;&\mbox{ in }\;\Omega,\\
\mathcal N_s\varphi=0&\mbox{ in } \;\mathcal O.
\end{cases}
\end{equation}
Then $\varphi=0$ in $\RR^N$.
\end{theorem}

\begin{proof}
We prove the result in two steps.

{\bf Step 1}. First we define the space
\begin{align*}
\mathbb W:=\Big\{(U_g)|_{\Omega}:\; U_g=\mathbb Pg,\;\; g\in\mathcal D(\mathcal O)\Big\}.
\end{align*}
We claim that $\mathbb W$ is dense in $L^2(\Omega)$. Indeed, by the Hahn-Banach theorem, it is sufficient to show that if $f\in L^2(\Omega)$ satisfies 
\begin{equation}\label{eq-v}
\int_{\Om}fw\;dx=0\;\;\mbox{ for all }\; w\in\mathbb W, 
\end{equation}
then $f\equiv 0$ in $\Omega$. Let $f$ satisfy \eqref{eq-v}. Then
\begin{equation*}
\int_{\Omega}f\mathbb Pg\;dx=0,\;\;g\in\mathcal D(\mathcal O).
\end{equation*}
Let $\mathbb P|_{\Omega}$ be the restriction of $\mathbb P$ on $\Omega$. That is, $(\mathbb P|_{\Omega})g=(\mathbb Pg)|_{\Omega}$. We show that the formal adjoint of $\mathbb P|_{\Omega}$ is given for $g\in\mathcal D(\mathcal O)$ by
\begin{equation}\label{e53}
\int_{\Omega}f\mathbb (Pg)|_{\Omega}\;dx=-\frac{C_{N,s}}{2}\int_{\RR^N}\int_{\RR^N}\frac{(V^f(x)-V^f(y))(g(x)-g(y))}{|x-y|^{N+2s}}\;dxdy,
\end{equation}
where we recall that $V^f\in\ W_0^{s,2}(\bOm)$ is the unique solution of \eqref{EDP-D}.
We notice that \eqref{e53} is equivalent to
\begin{align*}
-\frac{C_{N,s}}{2}\int_{\RR^N}\int_{\RR^N}\frac{(V^f(x)-V^f(y))(g(x)-g(y))}{|x-y|^{N+2s}}\;dxdy=\int_{\Omega}f w\;dx,\;\forall\; w\in W_0^{s,2}(\bOm).
\end{align*}
Let $g\in\mathcal D(\mathcal O)$ and $U_g=\mathbb Pg\in W^{s,2}(\RR^N)$. Then $(U_g-g)\in W_0^{s,2}(\bOm)$. Therefore, using that $\mathbb Pg=U_g$ is the solution of \eqref{EDP}, $g=0$ in $\Omega$ and $V^f\in W_0^{s,2}(\bOm)$ is the solution of \eqref{EDP-D} we get that
\begin{align*}
\int_{\Omega}f\mathbb (Pg)|_{\Omega}\;dx=&\int_{\Omega}f(U_g-g)\;dx\\
=&\frac{C_{N,s}}{2}\int_{\RR^N}\int_{\RR^N}\frac{(V^f(x)-V^f(y))((U_g-g)(x)-(U_g-g)(y))}{|x-y|^{N+2s}}\;dxdy\\
=&\frac{C_{N,s}}{2}\int_{\RR^N}\int_{\RR^N}\frac{(V^f(x)-V^f(y))(U_g(x)-U_g(y))}{|x-y|^{N+2s}}\;dxdy\\
&-\frac{C_{N,s}}{2}\int_{\RR^N}\int_{\RR^N}\frac{(V^f(x)-V^f(y))(g(x)-g(y))}{|x-y|^{N+2s}}\;dxdy\\
=&-\frac{C_{N,s}}{2}\int_{\RR^N}\int_{\RR^N}\frac{(V^f(x)-V^f(y))(g(x)-g(y))}{|x-y|^{N+2s}}\;dxdy.
\end{align*}
In the last equality we have used Definition \ref{def-sol} since $U_g$ is the solution of \eqref{EDP}. We have shown \eqref{e53}.

Combining \eqref{eq-v} and \eqref{e53} we get that
\begin{align*}
\frac{C_{N,s}}{2}\int_{\RR^N}\int_{\RR^N}\frac{(V^f(x)-V^f(y))(g(x)-g(y))}{|x-y|^{N+2s}}\;dxdy=0,\;\;g\in\mathcal D(\mathcal O).
\end{align*}
The preceding identity implies that
\begin{align*}
0=\int_{\RR^N}(-\Delta)^{\frac s2}V^f(-\Delta)^{\frac s2}g\;dx=\int_{\RR^N}g(-\Delta)^sV^f\;dx,\;\;g\in \mathcal D(\mathcal O).
\end{align*}
Since $g=0$ in $\RR^N\setminus\mathcal O$, the preceding identity implies that $V^f\in W_0^{s,2}(\bOm)$ satisfies
\begin{align*}
V^f=(-\Delta)^sV^f=0\;\;\mbox{ in }\;\mathcal O.
\end{align*}
It follows from \cite[Theorem 1.2]{GSU} that $V^f=0$. Thus $f=0$ and the claim is proved.

{\bf Step 2}. Now let $\lambda>0$ and let $\varphi\in W_0^{s,2}(\bOm)$ satisfy \eqref{e217}. Let $g\in\mathcal D(\mathcal O)$. Since $\mathcal N_s\varphi=0$ in $\mathcal O$ and $g=0$ in $(\Omc)\setminus\mathcal O$, it follows from \eqref{eqA9} that
\begin{align*}
0=\lambda\int_{\Omega}\varphi U_g\;dx+\int_{\Omc}g\mathcal N_s\varphi\;dx=\lambda \int_{\Omega}\varphi U_g\;dx.
\end{align*}
Since $\lambda>0$, this implies  that for every $U_g\in\mathbb W$,
\begin{align*}
\int_{\Omega}\varphi U_g\;dx=0.
\end{align*}
Since $\mathbb W$ is dense in $L^2(\Omega)$, it follows from the preceding identity that $\varphi=0$ in $\Omega$. Since $\varphi=0$ in $\Omc$, we have that $\varphi=0$ in $\RR^N$. The proof is finished.
\end{proof}



We conclude this section with the following remark.

\begin{remark}
{\em
 We mention the following facts.
\begin{enumerate}
\item Firstly, we notice that to prove the corresponding result of Theorem \ref{theo-27} for the Laplace operator or general second order elliptic operators, one usually uses the associated Pohozaev identity. Since the expression $\mathcal N_s\varphi$ does not appear in the Pohozaev identity for the fractional Laplace operator (see e.g. \cite{RS-Po2,RS-Po}), then this identity cannot be used to obtain Theorem \ref{theo-27}.

\item Secondly,  it has been shown in \cite[Proposition 4.2]{Ros} that if $\lambda>0$ and $\varphi\in D((-\Delta)_D^s)$ satisfy
\begin{equation*}
\begin{cases}
(-\Delta)_D^s\varphi=\lambda\varphi\;\;\;&\mbox{ in }\;\Omega,\\
\displaystyle\left.\frac{\varphi}{\rho^s}\right|_{\pOm}=0&\mbox{ on } \;\pOm,
\end{cases}
\end{equation*}
where $\rho(x):=\operatorname{dist}(x,\pOm)$, $x\in\Om$, then $u=0$ in $\RR^N$. The proof of this unique continuation property is also done by using the above mentioned Pohozaev identity for the fractional Laplacian.

\item Finally, we notice that even if the two notions of normal derivation, $\mathcal N_su$ and $\left.\frac{u}{\rho^s}\right|_{\pOm}$, of a function $u$, are different, at the limit they coincide in the sense that for all $u\in C^1(\bOm)\cap C_0^s(\Omega)$ and $v\in C_0^1(\RR^N)$, the following identities
\begin{align*}
\lim_{s\uparrow 1^-} \int_{\pOm}v\frac{u}{\rho^s}\;d\sigma=\lim_{s\uparrow 1^-}\int_{\RR^N\setminus\Om}v\mathcal N_su\;dx=\int_{\pOm}v\partial_\nu u\;d\sigma,
\end{align*}
hold, where $\partial_\nu u$ is the classical normal derivative of the function $u$. We refer to \cite[Section 5]{DRV} for more details.
\end{enumerate}
}
\end{remark}

\section{Some well-posedness results} \label{sec-proof-in}

In this section we study the existence, regularity and the representation of solutions to the systems \eqref{Int-eq} and \eqref{dual}. 
We start with \eqref{Int-eq}. Throughout the remainder of the paper, for $\beta>0$, $E_{\alpha,\beta}$ shall denote the Mittag-Leffler function defined in \eqref{mm}. We also mention that there are several references on abstract Cauchy problems of fractional order that give the existence of solutions of \eqref{Int-eq-2} and their representation in terms of the Mittag-Leffler functions. But for \eqref{Int-eq} there is no reference available. For this reason we will give the full proof. Throughout the following $(\varphi_n)_{n\in\NN}$ denotes the orthonormal basis of normalized eigenfunctions of $(-\Delta)_D^s$ associated with the eigenvalues $(\lambda_n)_{n\in\NN}$.

\subsection{Existence and representation of solutions to the system \eqref{Int-eq}}

We have the following result of existence and representation of solutions.

\begin{theorem}\label{theo-weak}
Let  $g\in \mathcal D((0,T)\times(\Omc))$. There exists a unique strong solution $u\in C^\infty([0,T]; W^{s,2}(\RR^N))$ of \eqref{Int-eq} which is given by
\begin{align}\label{rep-sol}
u(t,x)=-\sum_{n=1}^\infty\left(\int_0^t\Big(g(t-\tau,\cdot),\mathcal N_s\varphi_n\Big)_{L^2(\Omc)}\tau^{\alpha-1}E_{\alpha,\alpha}(-\lambda_n\tau^\alpha)\;d\tau\right)\varphi_n(x).
\end{align}
Moreover, the series in \eqref{rep-sol} converges in $C^m([0,T];W^{s,2}(\RR^N))$ and
\begin{align}\label{eq-ess}
&\|\partial_t^mu(t,\cdot)\|_{W^{s,2}(\RR^N)}\notag\\
\le& C\left(t^{\alpha(\gamma-1)+1}\|\partial_t^{m+1}g\|_{L^\infty((0,T);W^{s,2}(\Omc))}+\|\partial_t^mg\|_{W^{s,2}(\Omc)}\right)
\end{align}
for $m=0,1,2,\cdots,$ where $0<\gamma<\frac 14$ is a real number.
\end{theorem}

\begin{proof}
Let $g\in \mathcal D((0,T)\times(\Omc))$. We prove the theorem in several steps.

{\bf Step 1}. Firstly, we show uniqueness. Assume that \eqref{Int-eq} has two solutions $u_1$, $u_2$ and let $Z:=u_1-u_2$. Then $Z$ is a solution of the system
\begin{equation}\label{eq-0}
\begin{cases}
\mathbb D_t^\alpha Z+(-\Delta)^sZ=0\;\;&\mbox{ in }\;(0,T)\times\Omega,\\
Z=0 &\mbox{ in }\;(0,T)\times(\RR^N\setminus\Omega),\\
Z(0,\cdot)=0&\mbox{ in }\;\Omega.
\end{cases}
\end{equation}
This can be rewritten as the following Cauchy problem:
\begin{equation*}
\begin{cases}
\mathbb D_t^\alpha Z+(-\Delta)_D^s Z=0\;\; &\mbox{ in }\;(0,T)\times\Omega,\\
Z(0,\cdot)=0 &\mbox{ in }\;\Omega.
\end{cases}
\end{equation*}
Thus, it follows from \cite{KLW} that the unique strong solution of \eqref{eq-0} is given by $Z=0$. Hence, $u_1=u_2$ and we have shown uniqueness.

{\bf Step 2}. Secondly, we show the existence. We prove that the expression given in \eqref{rep-sol} satisfies \eqref{Int-eq}. Indeed, let $U_g$ be the unique solution of \eqref{EDP}. Since $g$ depends on $(t,x)$, then $U_g$ also depends on $(t,x)$.  Let $Y$ be a strong solution of 
\begin{equation}\label{Int-eq2}
\begin{cases}
\mathbb D_t^\alpha Y+(-\Delta)^sY=-\mathbb D_t^\alpha U_g\;\;&\mbox{ in }\;(0,T)\times\Omega,\\
Y=0 &\mbox{ in }\;(0,T)\times(\RR^N\setminus\Omega),\\
Y(0,\cdot)=0&\mbox{ in }\;\Omega.
\end{cases}
\end{equation}
Then clearly,
\begin{align*}
\mathbb D_t^\alpha(U_g+Y)+(-\Delta)^s(U_g+Y)=& \mathbb D_t^\alpha U_g+(-\Delta)^sU_g+\mathbb D_t^\alpha Y+(-\Delta)^sY\\
=&\mathbb D_t^\alpha U_g-\mathbb D_t^\alpha U_g=0.
\end{align*}
In addition $(U_g+Y)(0,\cdot)=0$ and $U_g+Y=g$ in $(0,T)\times (\Omc)$. Thus,
$u:=U_g+Y$ will solve \eqref{Int-eq}. Since $U_g\in \mathcal D((0,T); W^{s,2}(\RR^N))$, we have that $U_g(0,\cdot)=0$ and also $\mathbb D_t^\alpha U_g=\partial_t^\alpha U_g$. Since $\mathbb D_t^\alpha U_g\in \mathcal D((0,T); W^{s,2}(\RR^N))$, it follows from \cite{KLW} (see also \cite{GW-F})  that \eqref{Int-eq2} has a unique strong (classical) solution $Y$ given by
\begin{align*}
Y(t,x)=-\sum_{n=1}^\infty\left(\int_0^t\Big(\partial_t^\alpha U_g(t-\tau,\cdot),\varphi_n\Big)_{L^2(\Omega)}\tau^{\alpha-1}E_{\alpha,\alpha}(-\lambda_n\tau^\alpha)\;d\tau\right)\varphi_n(x).
\end{align*}
 Since $\partial_t^\alpha U_g\in C^{\infty}([0,T];W^{s,2}(\RR^N))$, it follows from \cite{KLW} or \cite{GW-F}  that  the function $Y\in C^{\infty}([0,T];W^{s,2}(\RR^N))$.
Using \eqref{eq33}, \eqref{eq34} and \eqref{eq35} we get that
\begin{align}\label{eq39}
Y(t,x)=&-\sum_{n=1}^\infty\left(\int_0^t\Big((I_{0,\tau}^{1-\alpha} \partial_\tau U_g)(t-\tau,\cdot),\varphi_n\Big)_{L^2(\Omega)}\tau^{\alpha-1}E_{\alpha,\alpha}(-\lambda_n\tau^\alpha)\;d\tau\right)\varphi_n(x)\notag\\
=&-\sum_{n=1}^\infty\left(\int_0^t\Big(\partial_\tau U_g(t-\tau,\cdot),\varphi_n\Big)_{L^2(\Omega)}I_{0,\tau}^{1-\alpha}\left(\tau^{\alpha-1}E_{\alpha,\alpha}(-\lambda_n\tau^\alpha)\right)\;d\tau\right)\varphi_n(x)\notag\\
=&-\sum_{n=1}^\infty\left(\int_0^t\Big(\partial_\tau U_g(t-\tau,\cdot),\varphi_n\Big)_{L^2(\Omega)}E_{\alpha,1}(-\lambda_n\tau^\alpha)\;d\tau\right)\varphi_n(x).
\end{align}
Since $U_g(0,\cdot)=0$, then integrating the right hand side of \eqref{eq39} by parts we get that
\begin{align*}
Y(t,x)=&\sum_{n=1}^\infty\left(\int_0^t\partial_\tau\Big( U_g(t-\tau,\cdot),\varphi_n\Big)_{L^2(\Omega)}E_{\alpha,1}(-\lambda_n\tau^\alpha)\;d\tau\right)\varphi_n(x)\\
=&-U_g(t,x)-\sum_{n=1}^\infty\left(\int_0^t\Big( U_g(t-\tau,\cdot),\varphi_n\Big)_{L^2(\Omega)}\partial_\tau\left(E_{\alpha,1}(-\lambda_n\tau^\alpha)\right)\;d\tau\right)\varphi_n(x).
\end{align*}
Using \eqref{Est-MLF2}, the preceding identity implies that
\begin{align}\label{eq-omc-1}
u(t,x):=&Y(t,x)+U_g(t,x)\notag\\
=&\sum_{n=1}^\infty\left(\int_0^t\Big( U_g(t-\tau,\cdot),\varphi_n\Big)_{L^2(\Omega)}\lambda_n\tau^{\alpha-1}\left(E_{\alpha,\alpha}(-\lambda_n\tau^\alpha)\right)\;d\tau\right)\varphi_n(x)\notag\\
=&\sum_{n=1}^\infty\left(\int_0^t\Big( U_g(t-\tau,\cdot),\lambda_n\varphi_n\Big)_{L^2(\Omega)}\tau^{\alpha-1}\left(E_{\alpha,\alpha}(-\lambda_n\tau^\alpha)\right)\;d\tau\right)\varphi_n(x).
\end{align}
It follows from \eqref{eqA9} that
\begin{align}\label{eq-omc}
\Big(U_g(t-\tau,\cdot),\lambda_n\varphi_n\Big)_{L^2(\Omega)}=-\Big(g(t-\tau,\cdot),\mathcal N_s\varphi_n\Big)_{L^2(\Omc)}.
\end{align}
Combining \eqref{eq-omc-1} and \eqref{eq-omc} we get that
\begin{align*}
u(t,x)=-\sum_{n=1}^\infty\left(\int_0^t\Big( g(t-\tau,\cdot),\mathcal N_s\varphi_n\Big)_{L^2(\Omc)}\tau^{\alpha-1}\left(E_{\alpha,\alpha}(-\lambda_n\tau^\alpha)\right)\;d\tau\right)\varphi_n(x).
\end{align*}
We have shown that $u$ given by \eqref{rep-sol} is the solution of \eqref{Int-eq}. 

{\bf Step 3}. Thirdly, we show that $u$ satisfies \eqref{eq-ess}. Let $0<\gamma<\frac 14$ be a real number. 
Using \eqref{eq39} we get that (notice that $Y(t,\cdot)\in D((-\Delta)_D^s)$)
\begin{align}\label{es-Y}
&\|(-\Delta)_D^sY(t,\cdot)\|_{L^2(\Omega)}\notag\\
=&\left\|\sum_{n=1}^\infty\lambda_n\left(\int_0^t\Big(\partial_t U_g(t-\tau,\cdot),\varphi_n\Big)_{L^2(\Om)}E_{\alpha,1}(-\lambda_n\tau^\alpha)\;d\tau\right)\varphi_n\right\|_{L^2(\Omega)}\notag\\
\le&\left\|\int_0^t\left(\sum_{n=1}^\infty\lambda_n\Big(\partial_t U_g(t-\tau,\cdot),\varphi_n\Big)_{L^2(\Om)}E_{\alpha,1}(-\lambda_n\tau^\alpha)\;d\tau\right)\varphi_n\right\|_{L^2(\Omega)}\notag\\
\le &\int_0^t\left\|\sum_{n=1}^\infty\lambda_n\Big(\partial_t U_g(t-\tau,\cdot),\varphi_n\Big)_{L^2(\Om)}E_{\alpha,1}(-\lambda_n\tau^\alpha)\;d\tau\varphi_n\right\|_{L^2(\Omega)}\notag\\
\le &\int_0^t\left(\sum_{n=1}^\infty\lambda_n^{2\gamma}\Big|\Big(\partial_t U_g(t-\tau,\cdot),\varphi_n\Big)_{L^2(\Om)}\Big|^2\cdot\Big|\lambda_n^{1-\gamma}E_{\alpha,1}(-\lambda_n\tau^\alpha)\Big|^2\right)^{\frac 12}\;d\tau.
\end{align}
Notice that $\partial_tU_g(t-\tau)=U_{\partial_tg}(t-\tau)$.
Using \eqref{27-2},  \eqref{27-1}, \eqref{New} and \eqref{E-Neu} we get from \eqref{es-Y} that there is a constant $C>0$ such that
\begin{align*}
\|Y(t,\cdot)\|_{W^{s,2}(\RR^N)}\le& \|(-\Delta)_D^sY(t,\cdot)\|_{L^2(\Omega)}\\
\le& C\int_0^t \left(\sum_{n=1}^\infty\lambda_n^{2\gamma}\Big|\Big(\partial_t U_g(t-\tau,\cdot),\varphi_n\Big)_{L^2(\Om)}\Big|^2\right)^{\frac 12}\tau^{\alpha(\gamma-1)}\;d\tau\\
=&C\int_0^t\|\partial_tU_g(t-\tau,\cdot)\|_{\mathbb V_{s,\gamma}}\tau^{\alpha(\gamma-1)}\;d\tau\\
\le&C\int_0^t\|\partial_tU_g(t-\tau,\cdot)\|_{W^{s,2}(\RR^N)}\tau^{\alpha(\gamma-1)}\;d\tau\\
\le &C\int_0^t\|\partial_tg(t-\tau,\cdot)\|_{W^{s,2}(\Omc)}\tau^{\alpha(\gamma-1)}\;d\tau\\
\le &C\|\partial_tg\|_{L^\infty((0,T);W^{s,2}(\Omc))}\int_0^t\tau^{\alpha(\gamma-1)}\;d\tau\\
\le &Ct^{\alpha(\gamma-1)+1}\|\partial_tg\|_{L^\infty((0,T);W^{s,2}(\Omc))}.
\end{align*}
Since $u=Y+U_g$, it follows from the preceding estimate and \eqref{E-Neu} that
\begin{align*}
\|u(t,\cdot)\|_{W^{s,2}(\RR^N)}\le &\|Y(t,\cdot)\|_{W^{s,2}(\RR^N)}+\|U_g(t,\cdot)\|_{W^{s,2}(\RR^N)}\\
\le& C\left(t^{\alpha(\gamma-1)+1}\|\partial_tg\|_{L^\infty((0,T);W^{s,2}(\Omc))}+\|g(t,\cdot)\|_{W^{s,2}(\Omc)}\right).
\end{align*}
Proceeding by induction we get \eqref{eq-ess} for enery $m\in\NN$.

{\bf Step 4}. Now, we prove that the series \eqref{rep-sol} converges in $C^m([0,T];W^{s,2}(\RR^N))$. Since $U_g\in C^\infty([0,T];W^{s,2}(\RR^N))$, it is sufficient to show the convergence of $Y$ given in \eqref{eq39}. Let $l,k\in\NN$ with $l<k$ and $0<\gamma<\frac 14$. Recall that $Y(t,\cdot)\in D((-\Delta)_D^s)$. By a similar calculation as in Step 3, and using \eqref{27-2}, we get that there is a constant $C>0$ such that
\begin{align*}
&\left\|-\sum_{n=l}^k\left(\int_0^t\Big( \partial_tU_g(t-\tau,\cdot),\varphi_n\Big)_{L^2(\Omega)}E_{\alpha,1}(-\lambda_n\tau^\alpha)\;d\tau\right)\varphi_n\right\|_{W^{s,2}(\RR^N)}\\
\le&C\left\|\sum_{n=l}^k\left(\int_0^t\Big( \partial_tU_g(t-\tau,\cdot),\varphi_n\Big)_{L^2(\Omega)}E_{\alpha,1}(-\lambda_n\tau^\alpha)\;d\tau\right)\varphi_n\right\|_{D((-\Delta)_D^s)}\\
\le &C\left\|\sum_{n=l}^k\lambda_n\left(\int_0^t\Big( \partial_tU_g(t-\tau,\cdot),\varphi_n\Big)_{L^2(\Omega)}E_{\alpha,1}(-\lambda_n\tau^\alpha)\;d\tau\right)\varphi_n\right\|_{L^2(\Omega)}\\
\le&C\int_0^t\left(\sum_{n=l}^k\lambda_n^{2\gamma}\Big|\Big(\partial_tU_g(t-\tau,\cdot),\varphi_n\Big)_{L^2(\Omega)}\Big|^2\right)^{\frac 12}\tau^{\alpha(\gamma-1)}\;d\tau\\
\le &C\sup_{0\le t\le T}\left(\sum_{n=l}^k\lambda_n^{2\gamma}\Big|\Big(\partial_tU_g(t,\cdot),\varphi_n\Big)_{L^2(\Omega)}\Big|^2\right)^{\frac 12}\int_0^t\tau^{\alpha(\gamma-1)}\;d\tau\\
\le &Ct^{\alpha(\gamma-1)+1}\sup_{0\le t\le T}\left(\sum_{n=l}^k\lambda_n^{2\gamma}\Big|\Big(\partial_tU_g(t,\cdot),\varphi_n\Big)_{L^2(\Omega)}\Big|^2\right)^{\frac 12}.
\end{align*}
Since $\partial_tU_g\in C([0,T];\mathbb V_{s,\gamma})$, it follows from the preceding estimate that
\begin{align*}
\sup_{0\le t\le T}&\left\|-\sum_{n=l}^k\left(\int_0^t\Big( \partial_tU_g(t-\tau,\cdot),\varphi_n\Big)_{L^2(\Omega)}E_{\alpha,1}(-\lambda_n\tau^\alpha)\;d\tau\right)\varphi_n\right\|_{W^{s,2}(\RR^N)}\\
&\le CT^{\alpha(\gamma-1)+1}\sup_{0\le t\le T}\left(\sum_{n=l}^k\lambda_n^{2\gamma}\Big|\Big(\partial_tU_g(t,\cdot),\varphi_n\Big)_{L^2(\Omega)}\Big|^2\right)^{\frac 12}\rightarrow 0\;\;\mbox{ as }\; l,k\to\infty.
\end{align*}
We have shown that the series \eqref{eq39} converges in $W^{s,2}(\RR^N)$ uniformly in $t\in [0,T]$. Similarly, we can show the uniform convergence of the series
\begin{align*}
\partial_t^mY(t,\cdot)=-\sum_{n=1}^\infty\left(\int_0^t\Big( \partial_t^{m+1}U_g(t-\tau,\cdot),\varphi_n\Big)_{L^2(\Omega)}E_{\alpha,1}(-\lambda_n\tau^\alpha)\;d\tau\right)\varphi_n,
\end{align*}
for any $m\in\NN$. The proof of the theorem is finished.
\end{proof}

\subsection{Existence and representation of solutions to the dual system} Now, we consider \eqref{dual}.
We have the following existence and representation of solutions.

\begin{proposition}\label{Dual-theo-weak}
Let  $u_0\in L^2(\Omega)$. Then \eqref{dual} has a unique strong solution $v$ given by
\begin{align}\label{eq413}
v(t,x)=\sum_{n=1}^\infty \Big(u_0,\varphi_n\Big)_{L^2(\Omega)}(T-t)^{\alpha-1}E_{\alpha,\alpha}(-\lambda_n(T-t)^\alpha)\varphi_n(x),
\end{align}
and there is a constant $C>0$ such that
\begin{align}\label{E1}
\left\|I_{t,T}^{1-\alpha}v(t,\cdot)\right\|_{L^2(\Om)}\le C\|u_0\|_{L^2(\Omega)},
\end{align}
and
\begin{align}\label{E2}
\left\|D_{t,T}^\alpha v(t,\cdot)\right\|_{L^2(\Omega)}\le C(T-t)^{-1}\|u_0\|_{L^2(\Omega)}.
\end{align}
Moreover, $v\in C([0,T);D((-\Delta)_D^s))\cap L^1((0,T);L^2(\Omega))$ and there is a constant $C>0$ such that
\begin{align}\label{E3}
\|v(t,\cdot)\|_{L^2(\Omega)}\le C(T-t)^{\alpha-1}\|u_0\|_{L^2(\Omega)}.
\end{align}
\end{proposition}

\begin{proof}
We prove the proposition in two steps.

{\bf Step 1}. First, we show uniqueness. Indeed, let $v$ be a solution of \eqref{dual} with $u_0=0$. Taking the inner product of \eqref{dual} with $\varphi_n$ and setting $v_n(t)=\left(v(t,\cdot),\varphi_n\right)_{L^2(\Omega)}$, we get that (given that the operator $(-\Delta)_D^s$ is selfadjoint)
\begin{equation}\label{du-sc}
D_{t,T}^\alpha v_n(t)=-\lambda_nv_n(t)\;\;\mbox{ for a.e. }\; t\in (0,T).
\end{equation}
Since $I_{t,T}^{1-\alpha}v\in C([0,T];L^2(\Omega))$, we have that $I_{t,T}^{1-\alpha}v_n=\left(I_{t,T}^{1-\alpha}v,\varphi\right)_{L^2(\Omega)}\in C[0,T]$ and
\begin{align*}
\left|I_{t,T}^{1-\alpha}v_n(t)\right|^2\le \sum_{n=1}^\infty\left|I_{t,T}^{1-\alpha}v_n(t)\right|^2\le \left\|I_{t,T}^{1-\alpha}v\right\|_{L^2(\Omega)}\to 0\;\mbox{ as }\; t\to T.
\end{align*}
This implies that
\begin{align}\label{du-sc-2}
I_{t,T}^{1-\alpha} v_n(T)=0.
\end{align}
Since \eqref{du-sc} with the final condition \eqref{du-sc-2} has zero as its unique solution (see e.g. \cite{Ba01}), it follows that $v_n(t)=0$ for $n=1,2,\ldots$. Since $(\varphi_n)$ is a complete system in $L^2(\Omega)$, we have that $v=0$ in $(0,T)\times\Omega$. The proof of the uniqueness is complete.

{\bf Step 2}. Second, we show the existence. Let $u_{0,n}:=(u_0,\varphi_n)_{L^2(\Omega)}$. For $1\le n\le m$ we set
\begin{align*}
v_m(t,x):=\sum_{n=1}^m u_{0,n}(T-t)^{\alpha-1}E_{\alpha,\alpha}(-\lambda_n(T-t)^\alpha)\varphi_n(x).
\end{align*}

\begin{enumerate}
\item[(i)] Let $v$ be given by \eqref{eq413}. We claim that $I_{t,T}^{1-\alpha}v\in C([0,T];L^2(\Omega))$. Integrating termwise, we get that
\begin{equation}\label{e419}
I_{t,T}^{1-\alpha}v_m(t,x)=\sum_{n=1}^mu_{0,n}E_{\alpha,1}(-\lambda_n(T-t)^\alpha)\varphi_n(x)
\end{equation}
in $L^2(\Om)$. Using \eqref{Est-MLF} and Lemma \ref{lem-INE}, we get that there is a constant $C>0$ such that for every $t\in [0,T]$ and $m,\tilde m\in\NN$ with $m>\tilde m$, we have
\begin{align*}
\left\|I_{t,T}^{1-\alpha}v_{\tilde m}-I_{t,T}^{1-\alpha}v_m\right\|_{L^2(\Omega)}^2=&2\sum_{n=\tilde m+1}^m\Big|u_{0,n}E_{\alpha,1}(-\lambda_n(T-t)^{\alpha})\Big|^2\\
\le &C\sum_{n=\tilde m+1}^m|u_{0,n}|^2\rightarrow 0\;\mbox{ as }\; \tilde m, m\to\infty.
\end{align*}
We have shown that the series 
\begin{align*}
\sum_{n=1}^\infty u_{0,n}E_{\alpha,1}(-\lambda_n(T-t)^\alpha)\varphi_n\rightarrow I_{t,T}^{1-\alpha} v(t,\cdot)\;\mbox{ in }\; L^2(\Omega),
\end{align*}
and that the convergence is uniform in $t\in [0,T]$. We have proved that $I_{t,T}^{1-\alpha}v\in C([0,T];L^2(\Omega))$. Using \eqref{Est-MLF} and Lemma \ref{lem-INE} again, we get that there is a constant $C>0$ such that
\begin{align*}
\left\|I_{t,T}^{1-\alpha}v(t,\cdot)\right\|_{L^2(\Omega)}^2\le C\|u_0\|_{L^2(\Omega)}^2.
\end{align*}
This gives \eqref{E1}.

\item[(ii)] We show that $D_{t,T}^\alpha v \in C([0,T);L^2(\Omega))$. This follows as in part (i) with the difference that here, the convergence is only uniform on compact subset of $[0,T)$. Since $D_{t,T}^\alpha v=-(-\Delta)_D^sv$, then using \eqref{IN-L2}, we get that there is a constant $C>0$ such that
\begin{align*}
\left\|D_{t,T}^\alpha v(t,\cdot)\right\|_{L^2(\Omega)}^2\le &2\sum_{n=1}^\infty|u_{0,n}|^2\Big|\lambda_n(T-t)^{\alpha-1}E_{\alpha,\alpha}(-\lambda_n(T-t)^\alpha)\Big|^2\\
\le&C(T-t)^{-2}\|u_0\|_{L^2(\Omega)}^2.
\end{align*}
Hence, $D_{t,T}^\alpha v \in C([0,T);L^2(\Omega))$. We have shown \eqref{E2} which also implies that $v\in C([0,T);D((-\Delta)_D^s))$.

\item[(iii)] It follows from \eqref{e419} that
\begin{align*}
I_{t,T}^{1-\alpha}v(T,\cdot)=\sum_{n=1}^\infty u_{0,n}\varphi_n=u_0.
\end{align*}
\end{enumerate}

Using \eqref{Est-MLF}, we get that there is a constant $C>0$ such that
\begin{align*}
\|v(t,\cdot)\|_{L^2(\Omega)}^2=&\left\|\sum_{n=1}^\infty u_{0,n}(T-t)^{\alpha-1}E_{\alpha,\alpha}(-\lambda_n(T-t)^\alpha)\varphi_n\right\|_{L^2(\Omega)}^2\\
\le &2\sum_{n=1}^\infty |u_{0,n}|^2\Big|(T-t)^{\alpha-1}E_{\alpha,\alpha}(-\lambda_n(T-t)^\alpha)\Big|^2\\
\le &C(T-t)^{2(\alpha-1)}\sum_{n=1}^\infty|u_{0,n}|^2=C(T-t)^{2(\alpha-1)}\|u_0\|_{L^2(\Omega)}^2,
\end{align*}
and we have shown \eqref{E3}. It follows from \eqref{E3} that
\begin{align*}
\int_0^T\|v(t,\cdot)\|_{L^2(\Omega)}\;dt\le C\|u_0\|_{L^2(\Omega)}\int_0^T(T-t)^{\alpha-1}\;dt=CT^\alpha\|u_0\|_{L^2(\Omega)}.
\end{align*}
Hence, $v\in L^1((0,T);L^2(\Omega))$.
The proof of the proposition is finished.
\end{proof}

\begin{lemma}\label{lem-anal}
Let $v$ be the unique strong solution of \eqref{dual}. Then the mapping $[0,T)\ni t\mapsto \mathcal N_sv(t,\cdot)\in L^2(\Omc)$ can be analytically extended to the half-plane $\Sigma_T:=\{z\in\CC:\; \operatorname{Re}(z)<T\}$.
\end{lemma}

\begin{proof}
 We recall that for every $t\in [0,T)$ fixed, we have that 
\begin{align*} 
 v(t,\cdot)\in D((-\Delta)_D^s)\hookrightarrow W_0^{s,2}(\bOm)\hookrightarrow W^{s,2}(\RR^N). 
\end{align*} 
 Hence, by Lemma \ref{lem-GSU}, $\mathcal N_sv(t,\cdot)$ exists and belongs to $L^2(\Omc)$.

 We claim that
\begin{align}\label{nor-sol}
\mathcal N_sv(t,\cdot)=\sum_{n=1}^\infty u_{0,n}(T-t)^{\alpha-1}E_{\alpha,\alpha}(-\lambda_n(T-t)^\alpha)\mathcal N_s\varphi_n,
\end{align}
and the series converges in $L^2(\Omc)$ for every $t\in [0,T)$. Let $\delta>0$ be fixed but arbitrary and let $t\in [0,T-\delta]$. Let $n,m\in\NN$ with $n>m$. Then, using the fact that $\mathcal N_sv(t,\cdot): D((-\Delta)_D^s)\subset W^{s,2}(\RR^N)\to L^2(\Omc)$ is bounded and \eqref{Est-MLF}, we get that there is a constant $C>0$ such that

 \begin{align*}
 &\left\Vert\sum_{n=m+1}^\infty u_{0,n}(T-t)^{\alpha-1}E_{\alpha,\alpha}(-\lambda_n(T-t)^\alpha)\mathcal N_s \varphi_n\right\Vert_{L^2(\Omc)}^2\\
 \le &C\left\Vert\sum_{n=m+1}^\infty u_{0,n}(T-t)^{\alpha-1}E_{\alpha,\alpha}(-\lambda_n(T-t)^\alpha) \varphi_n\right\Vert_{D((-\Delta)_D^s)}^2\\
 \le &C\sum_{n=m+1}^\infty |u_{0,n}|^2\Big|\lambda_n(T-t)^{\alpha-1}E_{\alpha,\alpha}(-\lambda_n(T-t)^\alpha)\Big|^2\\
 \le &C\sum_{n=m+1}^\infty |u_{0,n}|^2|T-t|^{-2}
 \le C\delta^{-2}\sum_{n=m+1}^\infty |u_{0,n}|^2\to 0\;\mbox{ as }\; m\to\infty.
 \end{align*}
We have shown that $\mathcal N_s v$ is given by \eqref{nor-sol} and the series is convergent in $L^2(\Omc)$ uniformly in any compact subset of $[0,T)$ and the claim is proved.

Since $E_{\alpha,\alpha}(-\lambda_nz)$ is an entire function, it follows that the function
\begin{align*}
(T-t)^{\alpha-1}E_{\alpha,\alpha}(-\lambda_n(T-t)^\alpha)
\end{align*}
 can be analytically extended to $\Sigma_T$. This implies that the function

\begin{align*}
 \sum_{n=1}^mu_{0,n}(T-z)^{\alpha-1}E_{\alpha,\alpha}(-\lambda_n(T-z)^\alpha)\mathcal N_s \varphi_n
\end{align*}
is analytic in $\Sigma_T$. Let $\delta>0$ be fixed but otherwise arbitrary. Let $z\in\CC$ satisfy $\mbox{Re}(z)\le T-\delta$. Then proceeding as above we get that
 \begin{align*}
 &\left\Vert\sum_{n=m+1}^\infty u_{0,n}(T-z)^{\alpha-1}E_{\alpha,\alpha}(-\lambda_n(T-z)^\alpha)\mathcal N_s\varphi_n\right\Vert_{L^2(\Omc)}^2\\
 \le &C\delta^{-2}\sum_{n=m+1}^\infty |u_{0,n}|^2\to 0\;\mbox{ as }\; m\to\infty.
 \end{align*}
 We have shown that
 \begin{align}\label{form-nor}
\mathcal N_sv(z,\cdot)=\sum_{n=1}^\infty u_{0,n}(T-z)^{\alpha-1}E_{\alpha,\alpha}(-\lambda_n(T-z)^\alpha)\mathcal N_s\varphi_n,
 \end{align}
and the series is uniformly convergent in any compact subset of $\Sigma_T$. Hence, $\mathcal N_sv$ given by \eqref{form-nor} is also analytic in $\Sigma_T$. 
The proof of the lemma is finished.
\end{proof}

\section{Proof of the main results}\label{sec-proof-res}

Now we give the proof of our main results.

\begin{proof}[\bf Proof of Theorem \ref{theo-cont-pro}]
Assume that $\mathcal N_s v=0$ in $(0,T)\times\mathcal O$. Since $\mathcal N_s v:[0,T)\to L^2(\Omc)$ can be analytically extended to $\Sigma_T$ (by Lemma \ref{lem-anal}), it follows that for $(t,x)\in(-\infty,T)\times\mathcal O$, 
\begin{align}\label{Exp-v}
\mathcal N_s v(t,x)
=\sum_{n=1}^\infty u_{0,n}(T-t)^{\alpha-1}E_{\alpha,\alpha}(-\lambda_n(T-t)^\alpha)\mathcal N_s \varphi_n(x)=0.
\end{align}
Let $\{\lambda_k\}_{k\in\NN}$ be the set of all eigenvalues of $(-\Delta)_D^s$ and $\{\psi_{k_j}\}_{1\le j\le m_k}$  an orthonormal basis for $\mbox{ker}(\lambda_k-(-\Delta)_D^s)$. Then \eqref{Exp-v} can be rewritten for $(t,x)\in(-\infty,T)\times\mathcal O$ as
\begin{align}\label{e156}
\mathcal N_s v(t,x)
=\sum_{k=1}^\infty \left(\sum_{j=1}^{m_k}u_{0,k_j}\mathcal N_s\psi_{k_j}(x)\right)(T-t)^{\alpha-1}E_{\alpha,\alpha}(-\lambda_k(T-t)^\alpha))=0.
\end{align}

Let $z\in\CC$ with $\mbox{Re}(z):=\eta>0$ and $m\in\NN$. Since $\psi_{k_j}$, $1\le j\le m_k$, $1\le k\le m$, are orthonormal, then using Lemma \ref{lem-INE} and the fact that $\mathcal N_s : \mathbb V_{s,1-\gamma}\subset W^{s,2}(\RR^N)\to L^2(\Omc)$ is bounded for every $0<\gamma< \frac 14 $,
we get that there is a constant $C>0$ such that

\begin{align}\label{M2}
 &\left\Vert\sum_{k=1}^m\left(\sum_{j=1}^{m_k}u_{0,k_j}\mathcal N_s\psi_{k_j}\right)e^{z(t-T)}(T-t)^{\alpha-1}E_{\alpha,\alpha}(-\lambda_k(T-t)^\alpha)\right\Vert_{L^2(\Omc)}^2\notag\\
\le &C\sum_{k=1}^\infty \left(\sum_{j=1}^{m_k}|u_{0,k_j}|^2\right)e^{2\eta(t-T)}\Big|\lambda_k^{1-\gamma}(T-t)^{\alpha-1}E_{\alpha,\alpha}(-\lambda_k(T-t)^\alpha)\Big|^2\notag\\
 \le& Ce^{2\eta(t-T)}(T-t)^{2(\alpha\gamma-1)}\|u_0\|_{L^2(\Omega)}^2.
\end{align}

Let
\begin{align*}
v_m(t,\cdot):=\sum_{k=1}^m\left(\sum_{j=1}^{m_k}u_{0,k_j}\mathcal N_s\psi_{k_j}\right)e^{z(t-T)}(T-t)^{\alpha-1}E_{\alpha,\alpha}(-\lambda_k(T-t)^\alpha).
\end{align*}

It follows from  \eqref{M2} that for every $0<\gamma< \frac 14$,
\begin{align}\label{nporm-2}
\|v_m(t,\cdot)\|_{L^2(\Omc)}
\le Ce^{\eta(t-T)}(T-t)^{\alpha\gamma-1}\|u_0\|_{L^2(\Omega)}.
\end{align}
The right-hand side of \eqref{nporm-2} is integrable over $t\in (-\infty,T)$ and 
\begin{align*}
\left(\int_{-\infty}^Te^{\eta(t-T)}(T-t)^{\alpha\gamma-1}\;dt\right)\|u_0\|_{L^2(\Omega)}=\frac{\Gamma(\alpha\gamma)}{\eta^{\alpha\gamma}}\|u_0\|_{L^2(\Omega)}.
\end{align*}

By the Lebesgue dominated convergence theorem, we get that 
\begin{align}\label{e158}
&\int_{-\infty}^Te^{z(t-T)}\sum_{k=1}^\infty \left(\sum_{j=1}^{m_k}u_{0,k_j}\mathcal N_s\psi_{k_j}\right)(T-t)^{\alpha-1}E_{\alpha,\alpha}(-\lambda_k(T-t)^\alpha)\;dt\notag\\
=&\sum_{k=1}^\infty \sum_{j=1}^{m_k}\frac{u_{0,k_j}}{z^\alpha+\lambda_k}\mathcal N_s\psi_{k_j},\;\;x\in\Omc,\;\;\mbox{Re}(z)>0.
\end{align}
In \eqref{e158}, we have used  that
\begin{align*}
\int_{-\infty}^Te^{z(t-T)}(T-t)^{\alpha-1}E_{\alpha,\alpha}(-\lambda_k(T-t)^\alpha)dt=\frac{1}{z^\alpha+\lambda_k},
\end{align*}
which follows from a change of variable and \eqref{lap-ml}.
It follows from \eqref{e156}, \eqref{e158} and the assumption, that 
\begin{align*}
\sum_{k=1}^\infty \sum_{j=1}^{m_k}\frac{u_{0,k_j}}{z^\alpha+\lambda_k}\mathcal N_s\psi_{k_j}(x)=0,\;\;x\in\mathcal O,\;\mbox{Re}(z)>0.
\end{align*}
Letting $\eta:=z^\alpha$, we have shown that
\begin{align}\label{e416}
\sum_{k=1}^\infty \sum_{j=1}^{m_k}\frac{u_{0,k_j}}{\eta+\lambda_k}\mathcal N_s\psi_{k_j}(x)=0,\;\;x\in\mathcal O,\;\mbox{Re}(\eta)>0.
\end{align}
Using the analytic continuation in $\eta$,  we get that \eqref{e416} holds for every $\eta\in\CC\setminus\{-\lambda_k\}_{k\in\NN}$. Taking a suitable small circle about $-\lambda_l$ and not including $\{-\lambda_k\}_{k\ne l}$ and integrating \eqref{e416} over that circle we get that 
\begin{align}\label{W}
\sum_{j=1}^{m_l}u_{0,l_j}\mathcal N_s\psi_{l_j}=0\;\;\mbox{ in }\;\mathcal O.
\end{align}
Let $w_l:= \sum_{j=1}^{m_l}u_{0,l_j}\psi_{l_j}$.
It follows from \eqref{W} that $\mathcal N_s w_l=0$ in $\mathcal O$. We have shown that for every $l$, $w_l$ solves the exterior nonlocal Neumann problem:
\begin{align*}
(-\Delta)_D^sw_l=\lambda_lw_l\;\;\mbox{  in }\; \Omega\;\;\mbox{ and }\;\mathcal N_s w_l=0 \;\mbox{ in }\;\mathcal O.
\end{align*}
It follows from Theorem \ref{theo-27}  that $w_l=0$ in $\Omega$ for every $l$. Since $\{\psi_{l_j}\}_{1\le j\le m_k}$ is linearly independent in $L^2(\Omega)$, we get that $(u_0,\psi_{l_j})_{L^2(\Omega)}=0$ for $1\le j\le m_k$, $k\in\NN$. Therefore,
 $u_0=0$ and we have shown that $v=0$ in $(0,T)\times \Omega$.  
The proof of the theorem is finished.
\end{proof}

Now, we prove our second main result.

\begin{proof}[\bf Proof of Theorem \ref{theo-appro-con}]

Let $g\in \mathcal D((0,T)\times\mathcal O)$. Recall that by Remark \ref{rem-33}, it suffices to prove that  \eqref{Int-eq} is approximately controllable.  Indeed, let $u$ be the unique strong solution of \eqref{Int-eq}  and $v$ the unique strong solution of \eqref{dual} with $u_0\in L^2(\Omega)$. First, it follows from Theorem \ref{theo-weak} that $\mathbb D_t^\alpha u$, $(-\Delta)^su\in L^\infty((0,T);L^2(\Omega))$. Second, it follows from Proposition \ref{Dual-theo-weak} that $v\in L^1((0,T);L^2(\Omega))$. Moreover, we have that $u(T,\cdot), I_{t,T}^{1-\alpha}v(T,\cdot)\in L^2(\Omega)$. Integrating by parts (by using \eqref{IP01}) on $(0,T-\delta)$ for $\delta>0$ small and taking the limit as $\delta\downarrow 0$ if necessary, and using  \eqref{Int-Part} we get that
\begin{align}\label{int-del-2}
0=&\int_0^{T}\int_{\Omega}\left(\mathbb D_t^\alpha u+(-\Delta)^su\right)v\;dxdt\notag\\
=&\int_0^{T}\int_{\Omega}v\mathbb D_t^\alpha u\;dxdt+\int_0^{T}\int_{\Omega}v(-\Delta)^su\;dxdt \notag\\
=&\int_0^{T}\int_{\Omega}u D_{t,T}^\alpha v\;dxdt+\Big(u(T,\cdot), I_{t,T}^{1-\alpha}v(T,\cdot)\Big)_{L^2(\Omega)}\notag\\
&+\int_0^{T}\int_{\Omega}u(-\Delta)^sv\;dxdt+\int_0^T\int_{\Omc}\Big(u\mathcal N_sv-v\mathcal N_su\Big)\;dx\;dt.
\end{align}
 It follows from \eqref{int-del-2} and \eqref{dual} that
\begin{align*}
0=&\int_0^{T}\int_{\Omega}\left(D_{t,T}^\alpha v+(-\Delta)^s v\right)u\;dxdt+\Big(u(T,\cdot), I_{t,T}^{1-\alpha}v(T,\cdot)\Big)_{L^2(\Omega)}\\
&+\int_0^T\int_{\Omc}g\mathcal N_s v\;dx\;dt\notag\\
=&\Big(u(T,\cdot), I_{t,T}^{1-\alpha}v(T,\cdot)\Big)_{L^2(\Omega)}+\int_0^T\int_{\Omc}g\mathcal N_s v\;dx\;dt.
\end{align*}
We  have shown  that
\begin{align}\label{eq41}
\int_{\Om}u(T,x)u_0(x)\;dx +\int_0^T\int_{\Omc}g\mathcal N_s v\;dx\;dt=0.
\end{align}
To prove that the set $\{(u(T,\cdot):\; g\in \mathcal D((0,T)\times\mathcal O)\}$ is dense in $L^2(\Omega)$, we have to show that if $u_0\in  L^2(\Omega)$ is such that
\begin{align}\label{eq42}
\int_{\Om}u(T,x)u_0(x)\;dx=0,
\end{align}
for any $g\in\mathcal D((0,T)\times\mathcal O)$, then $u_0=0$. Indeed, let $u_0$ satisfy \eqref{eq42}. It follows from \eqref{eq41} and \eqref{eq42} that
\begin{align*}
\int_0^T\int_{\Omc}g\mathcal N_s v\;dx\;dt=0,
\end{align*}
for any $g\in \mathcal D((0,T)\times\mathcal O)$. By the fundamental lemma of the calculus of variations, we have that
\begin{align*}
\mathcal N_sv=0\;\;\mbox{ in }\; (0,T)\times\mathcal O.
\end{align*}
It follows from Theorem \ref{theo-cont-pro} that $v=0$ in $(0,T)\times\Omega$.
Since the solution of \eqref{dual} is unique, we have that $u_0=0$ on $\Omega$. 
The proof of the theorem is finished.
\end{proof}

{\bf Acknowledgments}: The work of the author is partially supported by the Air Force Office of Scientific Research under the Award No: FA9550-15-1-0027.

\bibliographystyle{plain}
\bibliography{biblio}

\end{document}